%% file: main.tex
\documentclass{article}

\usepackage{PRIMEarxiv}

\input{math_commands.tex}   % theorem environments, math macros

\usepackage{amsmath}
\usepackage{mathtools}
\usepackage[utf8]{inputenc}
\usepackage[T1]{fontenc}
\usepackage{hyperref}
\usepackage{url}
\usepackage{booktabs}
\usepackage{amsfonts}
\usepackage{amssymb}
\usepackage{hyphenat}
\usepackage{microtype}
\usepackage{verbatim}
\usepackage{nicefrac}
\usepackage{fancyhdr}
\usepackage{graphicx}
\graphicspath{{media/}}
\usepackage{pgfplots}
\pgfplotsset{compat=1.17}
\usepackage{mathrsfs}
\usepackage{algorithm}
\usepackage{algorithmic}
\usepackage{subfig}
\usepackage[numbers]{natbib}

\title{Minimax Estimation of Kernel Stein Discrepancy: Trace versus Hilbert--Schmidt Scales}

\author{
  Davit Gogolashvili \\
  Weierstrass Institute\\
  Anton-Wilhelm-Amo-Straße 39, 10117 Berlin  \\
  \texttt{davit.gogolashvili@wias-berlin.de}
}

\begin{document}
\maketitle

%% ====================================================================
\begin{abstract}
Kernel Stein Discrepancy (KSD) compares a sample to a fixed target
distribution known only through its score, and is widely used for
goodness-of-fit testing, sample quality assessment, and approximate
inference.  We study the estimation of $\operatorname{KSD}(P_0,P)$ from
$n$ independent observations and identify the sharp spectral constant
governing the minimax risk: it is the Hilbert--Schmidt norm of the Stein
covariance operator $C_\star$, giving the minimax scale
$\sqrt{\|C_\star\|_{\mathrm{HS}}/n}$.  This scale is attained by
the positive-part square-root U-statistic, whereas the standard plug-in
V-statistic remains at the trace scale
$\sqrt{\operatorname{tr}(C_\star)/n}$ and is therefore suboptimal by the
fourth root of the effective rank of $C_\star$; for a Gaussian target
with a fixed-bandwidth Gaussian kernel this factor is exponential in the
dimension. 
\end{abstract}

%% ====================================================================
\section{Introduction}\label{sec:intro}

Comparing a sample to a target distribution is a basic task in statistics
and machine learning. In many applications the target density is known only
up to a normalizing constant. This happens, for example, for Bayesian
posteriors, energy-based models, and distributions produced by approximate
MCMC methods. The kernel Stein discrepancy (KSD) is useful in this setting
because it depends on the target \(P_0\) through the score
\(\nabla\log p_0\), and not through the normalizing constant.

Stein's method was introduced by \citet{stein1972bound} and then developed
as a general method for distributional approximation; see, for example,
\citet{stein1986approximate} and \citet{chen2011normal}. Its use for sample
quality assessment is more recent. \citet{gorham2015measuring} proposed a
computable Stein discrepancy for comparing exact, biased, and deterministic
sample sequences. The kernel version was introduced soon after by
\citet{liu2016kernelized} and \citet{chwialkowski2016kernel}, who combined
Stein identities with reproducing kernels and used the resulting KSD for
goodness-of-fit testing. \citet{gorham2017measuring} studied when KSD
controls convergence to the target and showed that the choice of kernel is
important for this property. For a recent overview of Stein's method in probabilistic 
inference and learning, see \citet{liu2026probabilistic}.

Several subsequent works developed KSD and related Stein discrepancies in
different directions.  On the algorithmic side, \citet{liu2016stein} used
Stein identities in Stein variational gradient descent, and recent work has
established finite-particle convergence rates in KSD and
Wasserstein metrics \citep{balasubramanian2025improved}. To reduce the quadratic cost of KSD, \citet{huggins2018random} introduced
random feature Stein discrepancies, and \citet{kalinke2024nystr} developed
a Nystr\"om approximation. On the 
testing side, \citet{jitkrittum2017linear} proposed a linear-time kernel Stein test, and
\citet{kanagawa2019kernel} extended Stein testing to latent-variable
models.  Stein discrepancies have also been used as optimization
objectives, for example in Stein points \citep{chen2018stein} and minimum
Stein discrepancy estimation \citep{barp2019minimum}. Since KSD is often
used as a goodness-of-fit statistic \citep{hagrass2026minimax,kalinke2026nystr}, minimax testing theory is also
relevant. In this direction, \citet{hagrass2026minimax} study minimax
goodness-of-fit testing with KSD at a fixed target \(P_0\). Their analysis
shows that ordinary unregularized KSD tests can be minimax suboptimal for
detecting alternatives separated in \(\chi^2(P,P_0)\), and that spectral
regularization is needed to attain the optimal testing boundary.

The closest comparison inside kernel methods is the minimax theory for
kernel mean embeddings and the maximum mean discrepancy (MMD). For a
reproducing kernel \(k\), the kernel mean embedding of \(P\) is
\(\mu_P=\mathbb E_{X\sim P}k(X,\cdot)\), and
\(\operatorname{MMD}(P,Q)=\|\mu_P-\mu_Q\|_{\mathcal H}\)
\citep{gretton2012kernel,muandet2017kernel}. \citet{tolstikhin2017minimax}
studied estimation of the embedding \(\mu_P\) itself. For
translation-invariant kernels on \(\mathbb R^d\), they proved that the
rate \(n^{-1/2}\) is minimax in the RKHS norm, and also in the
\(L^2(\mathbb R^d)\) norm, over discrete distributions and over
distributions with infinitely differentiable densities.

The scalar problem of estimating MMD was studied in
\citet{tolstikhin2016minimax}. For radial universal kernels on
\(\mathbb R^d\), they proved lower bounds matching the empirical MMD
estimator and its U-statistic variant, with rate
\(n^{-1/2}+m^{-1/2}\). Their sharp lower bound uses two fuzzy hypotheses
and removes a superfluous dimension dependence that appears if one argues
only through mean-embedding estimation. The KSD problem considered here
is closer to this scalar problem than to estimation of the whole mean
element. Indeed,
\[
    \operatorname{KSD}(P_0,P)
    =
    \big\|
        \mathbb E_{X\sim P}\xi_{P_0}(X)
    \big\|_{\mathcal H},
\]
where \(\xi_{P_0}\) is the Stein feature \ref{eq:steinfeature}. The main difference from the
usual MMD setting is that the target \(P_0\) is fixed and the Stein
feature depends on its score, hence it is typically unbounded.

The result closest to ours is \citet{cribeiro2025minimax}.  They establish
a minimax lower bound of order \(n^{-1/2}\) for KSD estimation, both for
Langevin--Stein KSD on \(\mathbb R^d\) and for KSDs on general domains.
Their minimax analysis is worst-case over both the target \(P_0\) and the
sampling distribution \(P\).  Their result settles the sample-size rate
and, together with existing upper bounds, implies rate optimality of the
usual V-statistic.  In this paper the target distribution is fixed, and the
worst case is taken only over the sampling distribution \(P\).  Our
analysis identifies the sharp spectral scale of estimation, which is not
visible from the optimal dependence on the sample size \(n\) alone.  In
particular, we show that the V-statistic is suboptimal in this spectral
sense, whereas the debiased U-statistic is minimax optimal.  For example,
for a Gaussian target with a Gaussian kernel, the gap between the two
scales grows exponentially with the dimension.

The relation to quadratic functional estimation is also worth mentioning.
A classical example is the problem of estimating \(\int f^2\), or more
generally quadratic functionals of an unknown density; see, for example,
\citet{birge1995estimation} and \citet{gine2008simple}, together with the
references therein. In that problem the density has to be reconstructed
from the observations. If one uses a Rosenblatt--Parzen estimator, the
bandwidth has to balance smoothing bias and variance. This is why the
optimal rate depends on the regularity of the unknown density, with the
usual elbow phenomenon at high regularity \citep{birge1995estimation}. The KSD problem is different in this respect. Once the target score and
the base kernel are fixed, the Stein kernel \(K_{P_0}\) is known. Thus
\(
  \operatorname{KSD}^2(P_0,P)
  =
  \mathbb E_{X,Y\sim P}K_{P_0}(X,Y)
\)
is the expectation of a known two-sample kernel. Consequently, there is no
density-smoothing bias for the unknown sampling distribution \(P\), and
the rates are always parametric.

\paragraph{Contributions.}
In this paper, we study the minimax optimal estimation of $\operatorname{KSD}(P_0,P)$. We are not merely
interested in convergence rates, but also in identifying the sharp spectral scale that governs the minimax scale.
The relevant  object is the covariance
operator of the Stein feature under the target,
\(
    C_\star
    =
    \mathbb E_{P_0}
    \bigl[
    \xi_{P_0}(X)\otimes \xi_{P_0}(X)
    \bigr]
\),
where \(\xi_{P_0}\) is the Langevin--Stein feature map (see
Section~\ref{sec:setup} for precise definitions). We show that the
 minimax scale is
\(
    \sqrt{\|C_\star\|_{\mathrm{HS}}/n}.
\)
Thus the  constant is governed by the Hilbert--Schmidt size of
\(C_\star\).

The lower bound is proved by a fuzzy-hypothesis construction along the
eigendirections of \(C_\star\). The matching upper bound is obtained by a
simple estimator: remove the diagonal terms from the usual squared KSD
estimator and take the positive square root. This is the square-root
U-statistic. Hence the minimax scale is achieved without estimating the
full Stein mean embedding.

We also quantify the behavior of the usual plug-in estimator. The plug-in V-statistic pays the trace scale because its
squared form keeps the diagonal terms. Its loss relative to the minimax
scale is the fourth root of the effective rank
\eqref{eq:effective rank}. This comparison is secondary to the
identification of the minimax scale, but it explains why removing the
diagonal terms matters. For the Gaussian target with a Gaussian kernel, we
compute the spectral quantities explicitly; for fixed bandwidth, the
resulting gap grows exponentially with dimension.

\paragraph{Notation.}
For a separable Hilbert space $\mathcal{H}$, $\langle\cdot,\cdot\rangle$
and $\|\cdot\|$ are its inner product and norm, $a\otimes b$ is the
rank-one operator $x\mapsto\langle b,x\rangle a$, and for a trace-class
operator $T$ we write $\operatorname{tr}(T)$ for its trace,
$\|T\|_{\mathrm{HS}}=\operatorname{tr}(T^2)^{1/2}$ (for self-adjoint $T$)
for its Hilbert--Schmidt norm, and $\|T\|_{\mathrm{op}}$ for its operator
norm. For probability measures, $\mathrm{TV}$, $\mathrm{KL}$ and $\chi^2$
denote total variation, Kullback--Leibler and chi-square divergences. We
write $a\lesssim b$ if $a\le Cb$ for an absolute constant $C$, and
$a\asymp b$ if $a\lesssim b\lesssim a$. For $u \in \R,$ $(u)_+=\max(u,0),$ and  
\(\lceil u\rceil\) denotes the smallest integer greater than or equal to \(u\).

\paragraph{Organization.}
Section~\ref{sec:setup} introduces the Stein feature map, the two
estimators, and the covariance operator $C_\star$ whose spectrum governs
the problem.  Section~\ref{sec:main} contains the main results.
Theorem~\ref{thm:upper} bounds the plug-in estimator by the trace scale
$\sqrt{\operatorname{tr}(C_\star)/n}$, and Theorem~\ref{thm:ustat} bounds
the debiased estimator by the Hilbert--Schmidt scale
$\operatorname{tr}(C_\star^2)^{1/4}/\sqrt n$.  Theorem~\ref{thm:lower}
gives a matching minimax lower bound, so that
Corollary~\ref{cor:optimal} identifies the Hilbert--Schmidt scale as
minimax optimal.  Proposition~\ref{prop:vlower} shows that the plug-in is
also bounded from below by the trace scale, and
Corollary~\ref{cor:gap} quantifies its loss as the fourth root of the
effective rank of $C_\star$.  Section~\ref{sec:gaussian} computes all
spectral quantities for the Gaussian target with a Gaussian kernel, where
the loss is exponential in the dimension for fixed bandwidth and
polynomial for bandwidth scaled with the dimension.  Proofs of the lower
bound and of the Gaussian calculations are given in the appendix.

\section{Setup}\label{sec:setup}

Let $P_0$ be the target distribution on
$\mathcal X\subseteq\mathbb R^d$. We assume that $P_0$ has a continuously
differentiable, everywhere positive density $p_0$, and we write
\[
  s_0(x)=\nabla\log p_0(x)
\]
for its score. The key feature of the KSD is that it depends on the target
only through this score. Hence the normalizing constant of $p_0$ is never
needed.

Let $k\in C^{1,1}(\mathcal X\times\mathcal X)$ be a positive definite
kernel with RKHS $\mathcal H_k$. We use the vector-valued Hilbert space
\(
  \mathcal H:=\mathcal H_k^d .
\)
The Langevin--Stein feature map associated with $P_0$ is the
$\mathcal H$-valued function
\begin{equation}\label{eq:steinfeature}
    \xi_{P_0}(x)
  :=
  s_0(x)k(x,\cdot)+\nabla_x k(x,\cdot).
\end{equation}
This feature map is the basic object of the paper: the KSD will be the norm
of the average of these features.

The inner product of two Stein features defines the Stein kernel
\[
  K_{P_0}(x,y)
  :=
  \left\langle
  \xi_{P_0}(x),\xi_{P_0}(y)
  \right\rangle_{\mathcal H}.
\]
Expanding this inner product gives the familiar closed form
\begin{equation}\label{eq:steinkernel}
\begin{aligned}
  K_{P_0}(x,y)
  &=
  s_0(x)^\top s_0(y)\,k(x,y)
  +s_0(x)^\top\nabla_y k(x,y)  \\
  &\quad
  +s_0(y)^\top\nabla_x k(x,y)
  +\nabla_x\!\cdot\!\nabla_y k(x,y).
\end{aligned}
\end{equation}
The expression above is useful computationally, because it allows the KSD
to be evaluated using only the score of the target and derivatives of the
base kernel.

For a sampling distribution $P$, define the Stein mean embedding by
\[
  \mu_P
  :=
  \mathbb E_P \xi_{P_0}(X),
\]
whenever the Bochner integral is well defined. The non-normalized kernel
Stein discrepancy is then
\[
  \operatorname{KSD}(P_0,P)
  :=
  \|\mu_P\|_{\mathcal H}.
\]
Thus the statistical problem studied here is not to estimate the whole
Hilbert-space vector $\mu_P$, but only its norm.

We impose the following basic integrability and Stein-identity assumption.

\begin{assumption}\label{ass:integrability}
Stein's identity holds at the target:
\[
  \mu_{P_0}
  =
  \mathbb E_{P_0}\xi_{P_0}(X)
  =
  0.
\]
Moreover,
\[
  \mathbb E_{P_0}K_{P_0}(X,X)<\infty,
\]
and for every sampling distribution $P$ considered below,
\[
  \mathbb E_P K_{P_0}(X,X)<\infty.
\]
\end{assumption}

The moment condition guarantees that the Stein mean embedding exists. It also gives the standard quadratic representation of the squared KSD:
if $X$ and $Y$ are independent draws from $P$, then
\begin{equation}\label{eq:ksd-identity}
  \operatorname{KSD}^2(P_0,P)
  =
  \|\mu_P\|_{\mathcal H}^2
  =
  \mathbb E_{X,Y\sim P}
  K_{P_0}(X,Y).
\end{equation}

Given i.i.d. samples $X_1,\dots,X_n\sim P$, the most common estimator
replaces the population Stein mean embedding by its empirical version
\[
  \widehat\mu_n
  :=
  \frac1n\sum_{i=1}^n \xi_{P_0}(X_i).
\]
The resulting plug-in estimator is
\[
  \widehat{\operatorname{KSD}}_V
  :=
  \|\widehat\mu_n\|_{\mathcal H}.
\]
Squaring it gives the V-statistic form
\begin{equation}\label{eq:vstat}
    \widehat{\operatorname{KSD}}_V^2
  =
  \frac1{n^2}
  \sum_{i,j=1}^n K_{P_0}(X_i,X_j).
\end{equation}

This estimator includes the diagonal terms $K_{P_0}(X_i,X_i)$. They are
harmless for consistency but, as we show below, responsible for the
suboptimal  behaviour of the plug-in estimator.

The debiased estimator removes these diagonal terms from the squared
estimator. Define
\begin{equation}\label{eq:ustat}
  U_n
  :=
  \frac1{n(n-1)}
  \sum_{i\neq j} K_{P_0}(X_i,X_j),
  \qquad
  \widehat{\operatorname{KSD}}_U
  :=
  \sqrt{(U_n)_+}.
\end{equation}
The statistic $U_n$ is unbiased for
$\operatorname{KSD}^2(P_0,P)$, while the positive square root turns it
back into an estimator of the KSD itself.

The  behaviour of these estimators is governed by the covariance
operator of the Stein feature under the target:
\begin{equation}\label{eq:cstar}
  C_\star
  :=
  \mathbb E_{P_0}
  \left[
  \xi_{P_0}(X)\otimes\xi_{P_0}(X)
  \right].
\end{equation}
Because $\mu_{P_0}=0$, this is both the second-moment operator and the
covariance operator of the Stein feature at the target. It is positive,
self-adjoint, and trace-class, with
\[
  \operatorname{tr}(C_\star)
  =
  \mathbb E_{P_0}K_{P_0}(X,X).
\]
Let $(\lambda_j,e_j)_{j\ge1}$ denote its nonzero eigenpairs, ordered as
\(
  \lambda_1\ge\lambda_2\ge\cdots>0.
\)
The two spectral quantities that drive the theory are
\[
  \operatorname{tr}(C_\star)
  =
  \sum_{j\ge1}\lambda_j,
  \qquad
  \operatorname{tr}(C_\star^2)
  =
  \sum_{j\ge1}\lambda_j^2.
\]
When $\operatorname{tr}(C_\star^2)>0$, define the effective rank
\begin{equation}\label{eq:effective rank}
  r_{\mathrm{eff}}(C_\star)
  :=
  \frac{\operatorname{tr}(C_\star)^2}
       {\operatorname{tr}(C_\star^2)}.
\end{equation}
This quantity measures how spread out the Stein spectrum is.

\section{Main Results}\label{sec:main}
In this section we present the main results.  We first prove upper bounds
for the plug-in V-statistic and the debiased U-statistic, then establish
the matching minimax lower bound.

\subsection{Upper bounds}
In this subsection we derive upper bounds for the two estimators introduced
in Section~\ref{sec:setup}.  We begin with the plug-in.

\begin{theorem}[V-statistic upper bound]\label{thm:upper}
Under Assumption~\ref{ass:integrability}, for every \(P\) with
\(\mathbb{E}_P K_{P_0}(X,X)<\infty\) and every \(n\ge1\),
\[
  \mathbb{E}_P\bigl|\widehat{\operatorname{KSD}}_V
    -\operatorname{KSD}(P_0,P)\bigr|
  \le
  \left(
  \frac{
  \mathbb{E}_P K_{P_0}(X,X)-\operatorname{KSD}^2(P_0,P)
  }{n}
  \right)^{1/2}
  \le
  \left(
  \frac{\mathbb{E}_P K_{P_0}(X,X)}{n}
  \right)^{1/2}.
\]
In particular, at the target,
\[
  \mathbb{E}_{P_0}\widehat{\operatorname{KSD}}_V
  \le
  \sqrt{\frac{\operatorname{tr}(C_\star)}{n}}.
\]
\end{theorem}

\begin{proof}
By the reverse triangle inequality and Jensen's inequality,
\[
  \mathbb{E}_P
  \bigl|
  \widehat{\operatorname{KSD}}_V-\operatorname{KSD}(P_0,P)
  \bigr|
  \le
  \left(
  \mathbb{E}_P\|\widehat\mu_n-\mu_P\|_{\mathcal H}^2
  \right)^{1/2}.
\]
Note that \(\xi_{P_0}(X_i)\) are i.i.d. with mean \(\mu_P\), therefore
\[
  \mathbb{E}_P\|\widehat\mu_n-\mu_P\|_{\mathcal H}^2
  =
  \frac1n
  \left(
  \mathbb E_P\|\xi_{P_0}(X)\|_{\mathcal H}^2
  -
  \|\mu_P\|_{\mathcal H}^2
  \right).
\]
Since
\(\|\xi_{P_0}(X)\|_{\mathcal H}^2=K_{P_0}(X,X)\) and
\(\|\mu_P\|_{\mathcal H}=\operatorname{KSD}(P_0,P)\), the claim follows.

\end{proof}

The bound of Theorem~\ref{thm:upper} controls the error of the whole Stein
mean element $\mu_P$ in $\mathcal H$, and not of its scalar norm
$\operatorname{KSD}(P_0,P)$ alone. This is not surprising, given the fact that the plug-in 
first estimates the Stein mean element \(\mu_P\) by the empirical mean \(\widehat\mu_n\) and then
takes its norm, so the reverse triangle inequality controls its error by
\(\|\widehat\mu_n-\mu_P\|_{\mathcal H}\). It thus treats the problem as
estimation of the whole vector \(\mu_P\), and not of its norm alone.

At the target this distinction is visible. Since \(\mu_{P_0}=0\), the bound reduces to
\(
  \mathbb E_{P_0}\|\widehat\mu_n\|_{\mathcal H}^2=\operatorname{tr}(C_\star)/n.
\)
The same trace appears in the squared plug-in statistic through its diagonal
terms $K_{P_0}(X_i,X_i)$, whose expectation at the target is
$\operatorname{tr}(C_\star)$.  The debiased statistic of~\eqref{eq:ustat}
removes these diagonal terms, and we analyze it next.

\begin{theorem}[U-statistic upper bound]\label{thm:ustat}
Let \(n\ge2\). Under Assumption~\ref{ass:integrability}, suppose in addition that
\[
  V_P
  :=
  \mathbb E_{P\otimes P}K_{P_0}(X,Y)^2
  <\infty,
\]
where \(X,Y\sim P\) are independent.  Then
\[
  \mathbb E_P U_n
  =
  \operatorname{KSD}^2(P_0,P).
\]
Moreover, with
\[
  \theta:=\operatorname{KSD}(P_0,P),
\]
we have
\begin{equation}\label{eq:ustat-var-bound}
  \operatorname{Var}_P(U_n)
  \le
  \frac{4\theta^2 V_P^{1/2}}{n}
  +
  \frac{2V_P}{n(n-1)}.
\end{equation}
Consequently,
\begin{equation}\label{eq:ustat-risk-holder}
  \mathbb E_P
  \left|
  \widehat{\operatorname{KSD}}_U
  -
  \operatorname{KSD}(P_0,P)
  \right|
  \le
  \bigl(\operatorname{Var}_P U_n\bigr)^{1/4}.
\end{equation}
If \(\theta>0\), then also
\begin{equation}\label{eq:ustat-risk-away}
  \mathbb E_P
  \left|
  \widehat{\operatorname{KSD}}_U
  -
  \operatorname{KSD}(P_0,P)
  \right|
  \le
  \frac{\bigl(\operatorname{Var}_P U_n\bigr)^{1/2}}{\theta}.
\end{equation}
In particular, at the target,
\begin{equation}\label{eq:ustat-null}
  \mathbb E_{P_0}
  \left|
  \widehat{\operatorname{KSD}}_U
  \right|
  \le
  \left(
  \frac{2\operatorname{tr}(C_\star^2)}{n(n-1)}
  \right)^{1/4}
  \le
  \sqrt{2}\,
  \sqrt{\frac{\|C_\star\|_{\mathrm{HS}}}{n}},
  \qquad n\ge2.
\end{equation}
\end{theorem}

The proof is given in Appendix~\ref{app:proof-ustat}. At the target the bound of Theorem~\ref{thm:ustat} is a consequence of
degeneracy.  Indeed, Stein's identity gives $\mu_{P_0}=0$, so the leading
term in the variance of $U_n$ vanishes and
$\operatorname{Var}_{P_0}U_n\asymp\operatorname{tr}(C_\star^2)/n^2$.  The
squared KSD estimator therefore fluctuates at the scale
$\operatorname{tr}(C_\star^2)^{1/2}/n$, and its square root at the scale
\(\sqrt{\|C_\star\|_{\mathrm{HS}}/n}\).

Theorems~\ref{thm:upper} and~\ref{thm:ustat} thus yield upper bounds
governed by two different spectral functionals of the same operator
$C_\star$: the plug-in estimator is controlled by
$\operatorname{tr}(C_\star)$ and the debiased estimator by
$\operatorname{tr}(C_\star^2)$, so that the two estimators fluctuate at the
target at the scales $\sqrt{\operatorname{tr}(C_\star)/n}$ and
\(\sqrt{\|C_\star\|_{\mathrm{HS}}/n}\) respectively.  Whether the
latter, smaller scale can be improved by any estimator is settled by the
lower bound of the next subsection.

\subsection{Lower bound}

The upper bound for the debiased estimator suggests that the relevant scale is 
\( \operatorname{tr}(C_\star^2)^{1/4}/\sqrt n. \) We now show that this scale 
cannot be improved in the minimax sense. The minimax lower bound we consider is over the 
target-normalized regularity class. More precisely,  for \(A \ge 1\), we define the following class of distributions
\[
  \mathcal P(A)
  :=
  \left\{
  P:
  \mathbb E_{P\otimes P}K_{P_0}(X,Y)^2
  \le
  A\,\operatorname{tr}(C_\star^2)
  \right\},
\]
where \(X,Y\sim P\) are independent. Note that \(\operatorname{tr}(C_\star^2)<\infty\) under
Assumption~\ref{ass:integrability}, therefore the normalization in the definition of \(\mathcal P(A)\) is
well-defined. 

\begin{remark}
One could equivalently introduce an absolute moment class
\[
  \mathcal P_0(A)
  :=
  \left\{
  P:
  \mathbb E_{P\otimes P}K_{P_0}(X,Y)^2
  \le A
  \right\}.
\]
In the present paper we use the target-normalized form
because our goal is to expose the spectral dependence on the target
distribution and the kernel.  At the target,
\(
  \mathbb E_{P_0\otimes P_0}K_{P_0}(X,Y)^2
  =
  \operatorname{tr}(C_\star^2),
\)
so the constant \(A\) measures how much larger the off-diagonal
Stein-kernel fluctuation is under \(P\) than under \(P_0\).  This
normalization also makes the class invariant under rescaling of the Stein
kernel.
\end{remark}

\begin{remark}[Bounded likelihood ratios]\label{rem: bounded likelihood}
The class \(\mathcal P(A)\) contains the usual bounded-likelihood-ratio
neighbourhoods of the target.  Indeed, if \(P\ll P_0\) and
\[
  0\le \frac{dP}{dP_0}\le L,
\]
then
\[
\begin{aligned}
  \mathbb E_{P\otimes P}K_{P_0}(X,Y)^2
  &=
  \mathbb E_{P_0\otimes P_0}
  \left[
  \frac{dP}{dP_0}(X)
  \frac{dP}{dP_0}(Y)
  K_{P_0}(X,Y)^2
  \right]  \\
  &\le
  L^2
  \mathbb E_{P_0\otimes P_0}K_{P_0}(X,Y)^2  \\
  &=
  L^2\operatorname{tr}(C_\star^2).
\end{aligned}
\]
Hence \(P\in\mathcal P(L^2)\).  The class \(\mathcal P(A)\) is more
general, however, because it does not require absolute continuity with
respect to \(P_0\).
\end{remark}

Let us comment on the choice of this class, as it may not be obvious at first glance why we choose
the off-diagonal moment rather than the diagonal moment class:
\[
  \mathcal D(A)
  :=
  \left\{
  P:  
  \mathbb E_P K_{P_0}(X,X)
  \le
  A\,\operatorname{tr}(C_\star)
  \right\}.      
\]
The diagonal class above is natural for estimating the full Stein mean embedding
\(\mu_P\), because it controls the second moment of the Stein
feature \(\xi_{P_0}(X)\). However, in $\operatorname{KSD}$ estimation,
we are only interested in the RKHS norm of the embedding, and not the embedding itself.
This distinction is important because the two problems have different statistical
difficulties.  Typically, estimating the full embedding is a harder problem than estimating the scalar norm alone.
Moreover, by the identity
\[
   \operatorname{KSD}^2(P_0,P)
  =\mathbb E_{P\otimes P}K_{P_0}(X,Y),
\]
it is an off-diagonal expectation of a Stein kernel.  Thus the diagonal
terms are not part of the population quantity itself;
they appear only when one estimates the full mean embedding and then takes
its norm, as in the plug-in/V-statistic estimator. For this reason, the
off-diagonal moment condition defining \(\mathcal P(A)\) is the natural
regularity condition for the scalar KSD estimation problem.

The next result gives the matching minimax lower bound over
\(\mathcal P(A)\).
\begin{theorem}[Minimax lower bound]\label{thm:lower}
Assume \(\operatorname{tr}(C_\star^2)>0\). For every fixed \(A\ge4\), there exists a constant
\(c>0\), independent of \(n\), such that for all sufficiently large
\(n\),
\[
  \inf_{\widehat T}
  \sup_{P\in\mathcal P(A)}
  \mathbb E_P
  \left|
  \widehat T-\operatorname{KSD}(P_0,P)
  \right|
  \ge
  c\,
  \sqrt{\frac{\|C_\star\|_{\mathrm{HS}}}{n}},
\]
where the infimum is over all estimators \(\widehat T\) based on \(n\) i.i.d. samples from \(P\).
\end{theorem}

The proof can be found in Appendix~\ref{app:proof-lower}. As we see,
the achievable optimal rates are parametric \(n^{-1/2}\), but most importantly,  Theorem~\ref{thm:lower} 
shows that no estimator can estimate \(\operatorname{KSD}\)\ uniformly 
over the regular class \(\mathcal P(A)\) at a scale smaller than
\(\sqrt{\|C_\star\|_{\mathrm{HS}}/n}\), up to constants.  The
Hilbert--Schmidt scale attained by \(\widehat{\operatorname{KSD}}_U\) is
therefore minimax optimal. This fixed-target lower bound refines the global 
minimax picture; cf. \citet{cribeiro2025minimax}, where the minimax risk is worst-case 
over both the target and the sampling distribution.

\subsection{Minimax optimality and the V/U gap}

We now combine the lower bound with the U-statistic upper bound.

\begin{corollary}[Minimax optimality]\label{cor:optimal}
Assume the conditions of Theorems~\ref{thm:lower} and~\ref{thm:ustat}.
For every fixed \(A\ge4\),
\[
  \inf_{\widehat{T}}
  \sup_{P\in\mathcal P(A)}
  \mathbb{E}_P
  \bigl|
  \widehat{T}-\operatorname{KSD}(P_0,P)
  \bigr|
  \;\asymp_A\;
  \frac{\operatorname{tr}(C_\star^2)^{1/4}}{\sqrt{n}}
  =
  \sqrt{\frac{\|C_\star\|_{\mathrm{HS}}}{n}}.
\]
Moreover, 
\(\widehat{\operatorname{KSD}}_U\) attains this rate.
\end{corollary}

\begin{proof}
The lower bound follows from Theorem~\ref{thm:lower}.  We prove the upper
bound for \(\widehat{\operatorname{KSD}}_U\).

Set
\[
  S:=\operatorname{tr}(C_\star^2),
  \qquad
  \theta:=\operatorname{KSD}(P_0,P).
\]
For \(P\in\mathcal P(A)\), we have \(V_P\le AS\).  Hence
Theorem~\ref{thm:ustat} gives
\[
  \operatorname{Var}_P(U_n)
  \lesssim_A
  \frac{S^{1/2}\theta^2}{n}
  +
  \frac{S}{n^2}.
\]
Using the two bounds in Theorem~\ref{thm:ustat}, we obtain
\[
  \mathbb E_P
  \left|
  \widehat{\operatorname{KSD}}_U-\theta
  \right|
  \le
  \min\left\{
  \operatorname{Var}_P(U_n)^{1/4},
  \frac{\operatorname{Var}_P(U_n)^{1/2}}{\theta}
  \right\},
\]
with the convention that the second term is omitted when \(\theta=0\).
Therefore,
\[
  \mathbb E_P
  \left|
  \widehat{\operatorname{KSD}}_U-\theta
  \right|
  \lesssim_A
  \min\left\{
  \left(
  \frac{S^{1/2}\theta^2}{n}
  +
  \frac{S}{n^2}
  \right)^{1/4},
  \frac{
  \left(
  \frac{S^{1/2}\theta^2}{n}
  +
  \frac{S}{n^2}
  \right)^{1/2}
  }{\theta}
  \right\}.
\]
If \(\theta\le S^{1/4}/\sqrt n\), the first term is
\(\lesssim_A S^{1/4}/\sqrt n\).  If
\(\theta>S^{1/4}/\sqrt n\), the second term is
\(\lesssim_A S^{1/4}/\sqrt n\).  Hence, uniformly over
\(P\in\mathcal P(A)\),
\[
  \mathbb E_P
  \left|
  \widehat{\operatorname{KSD}}_U
  -
  \operatorname{KSD}(P_0,P)
  \right|
  \lesssim_A
  \frac{\operatorname{tr}(C_\star^2)^{1/4}}{\sqrt n}.
\]
Combining this upper bound with Theorem~\ref{thm:lower} proves the result.
\end{proof}

Corollary~\ref{cor:optimal} identifies the Hilbert--Schmidt norm of
$C_\star$ as the spectral quantity governing the minimax risk of KSD
estimation, and shows that the debiased estimator attains it without
estimating the Stein mean embedding.

We now compare the plug-in estimator to this benchmark.  In Theorem~\ref{thm:upper} we have seen that the plug-in 
estimator is bounded above by the trace scale. Strictly speaking, this does not preclude the possibility that the plug-in is
minimax optimal, unless we show that the plug-in is also bounded below by the trace scale. The next result shows that this is indeed the case.

\begin{proposition}[Plug-in pinned to the trace scale]\label{prop:vlower}
Assume \(\operatorname{tr}(C_\star)>0\).  Under
Assumption~\ref{ass:integrability}, if
\(\mathbb{E}_{P_0}K_{P_0}(X,X)^2<\infty\), then for
\[
  n\ge
  n_0
  :=
  \left\lceil
  \frac{
  \mathbb{E}_{P_0}K_{P_0}(X,X)^2
  }{
  \operatorname{tr}(C_\star)^2
  }
  \right\rceil,
\]
we have
\[
  \frac{1}{2}
  \sqrt{\frac{\operatorname{tr}(C_\star)}{n}}
  \le
  \mathbb{E}_{P_0}\widehat{\operatorname{KSD}}_V
  \le
  \sqrt{\frac{\operatorname{tr}(C_\star)}{n}}.
\]
\end{proposition}

The plug-in estimator cannot benefit from the degeneracy at the target.
Even when the true KSD is zero, the empirical Stein mean \(\widehat\mu_n\)
has Hilbert-space variance \(\operatorname{tr}(C_\star)/n\), and the plug-in
takes its norm; its risk is therefore pinned to the trace scale.  The trace
scale is thus not only the upper bound of Theorem~\ref{thm:upper} but also a
lower bound for the plug-in itself.

Combining Proposition~\ref{prop:vlower} with the minimax rate gives
the precise loss of the plug-in estimator.

\begin{corollary}[Spectral suboptimality of the plug-in]\label{cor:gap}
Under the assumptions of Proposition~\ref{prop:vlower}, the plug-in's risk
at the target exceeds the minimax benchmark by
\[
  \frac{
  \mathbb{E}_{P_0}\widehat{\operatorname{KSD}}_V
  }{
  \operatorname{tr}(C_\star^2)^{1/4}/\sqrt{n}
  }
  \;\asymp\;
  \left(
  \frac{\operatorname{tr}(C_\star)^2}
       {\operatorname{tr}(C_\star^2)}
  \right)^{1/4}
  =
  r_{\mathrm{eff}}(C_\star)^{1/4}.
\]
\end{corollary}

The price of the plug-in, relative to the debiased estimator, is thus the
fourth root of the effective rank of \(C_\star\).  The plug-in is
sharp only when the spectrum is effectively low-dimensional; when the Stein
spectrum is spread over many directions, the diagonal bias of the
V-statistic makes it strictly suboptimal. Below we consider the example
where the effective rank can be calculated explicitly.

\subsection{Dimension dependence at a Gaussian target}\label{sec:gaussian}
Consider the case where the target is the
standard Gaussian \(P_0=N(0,I_d)\) with Gaussian reproducing kernel
\(k_\gamma(x,y)=\exp(-\gamma\|x-y\|^2)\).  In this case, the Stein kernel is
\[
  K_{P_0}(x,y)
  =
  e^{-\gamma\|x-y\|^2}
  \bigl[
    x^\top y+2\gamma d-(2\gamma+4\gamma^2)\|x-y\|^2
  \bigr],
\]
and hence
\[
  \operatorname{tr}(C_\star)
  =
  \mathbb E_{P_0}K_{P_0}(X,X)
  =
  (1+2\gamma)d .
\]
A direct calculation gives (see  Appendix~\ref{app:gauss} for the detailed derivation)
\[
  \operatorname{tr}(C_\star^2)
  =
  (1+8\gamma)^{-d/2}
  \frac{d}{(1+8\gamma)^2}
  \Bigl[
    (64\gamma^4+32\gamma^3+4\gamma^2)d
    +128\gamma^4+128\gamma^3+80\gamma^2+16\gamma+1
  \Bigr] \asymp_\gamma d^2(1+8\gamma)^{-d/2}.
\]
Considering this in the minimax-optimal scale of Corollary~\ref{cor:optimal} gives the minimax rates
\[
\frac{\operatorname{tr}(C_\star^2)^{1/4}}{\sqrt n}
\asymp_\gamma
\frac{d^{1/2}(1+8\gamma)^{-d/8}}{\sqrt n}.
\]
Notice that it exponentially \emph{decreases} in the dimension $d$.  The
reason is that the Stein kernel is degenerate away from the diagonal: for
independent $X,Y\sim P_0$ the squared distance $\|X-Y\|^2$ concentrates
around $2d$, so the Gaussian kernel treats distinct points as nearly
orthogonal and $K_{P_0}(X,Y)$ is exponentially small.

For fixed bandwidth \(\gamma>0\),  
the V/U gap satisfies 
\[
  r_{\mathrm{eff}}(C_\star)^{1/4}
  =
  \left(
  \frac{\operatorname{tr}(C_\star)^2}
       {\operatorname{tr}(C_\star^2)}
  \right)^{1/4}
  \asymp_\gamma
  (1+8\gamma)^{d/8}.
\]
Hence, for a fixed bandwidth, the plug-in loses an exponential factor in the
dimension. The picture changes if the bandwidth is scaled with the dimension. If
\(\gamma=\alpha/d\) for fixed \(\alpha>0\), then
\[
  \operatorname{tr}(C_\star^2)\asymp_\alpha d,
  \qquad
  \frac{\operatorname{tr}(C_\star^2)^{1/4}}{\sqrt n}
  \asymp_\alpha
  \frac{d^{1/4}}{\sqrt n},
\]
and
\[
  r_{\mathrm{eff}}(C_\star)^{1/4}
  \asymp_\alpha
  d^{1/4}.
\]
Thus bandwidth choice changes the dimension dependence of the V/U gap from exponential to a polynomial one.
\begin{figure}[t]
\centering
\includegraphics[width=\textwidth]{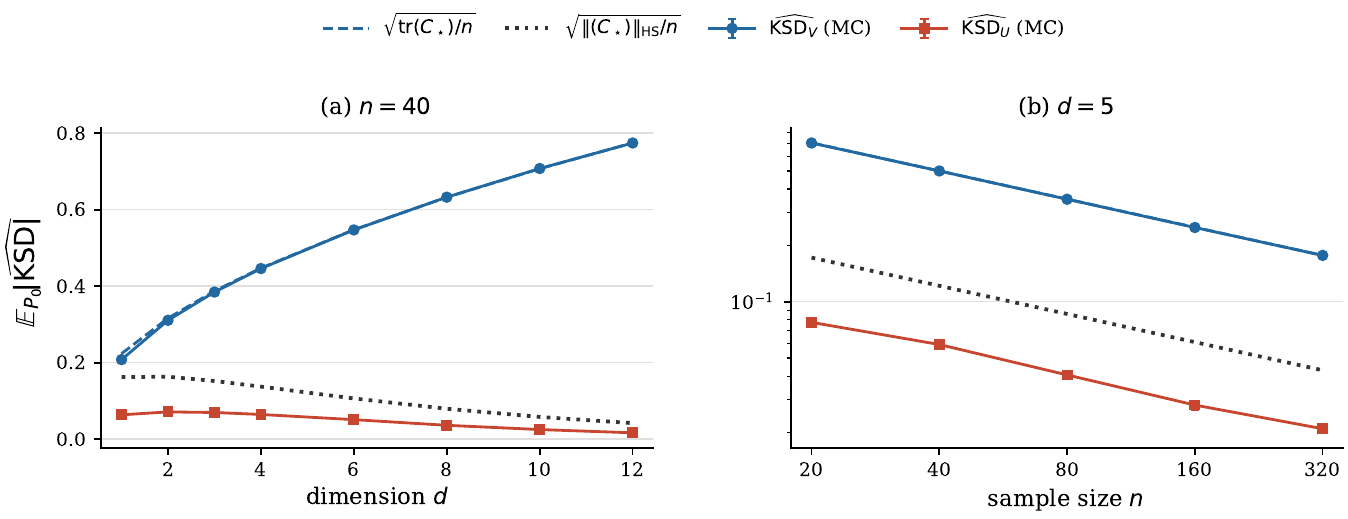}
\caption{Comparison of the empirical target risks with the trace and
Hilbert--Schmidt scales for \(P_0=N(0,I_d)\) and the Gaussian kernel with
\(\gamma=\tfrac12\).  \textbf{(a)} The four quantities are plotted against
the dimension \(d\) for fixed \(n=40\).  The plug-in risk increases with
the trace scale, whereas the square-root U-statistic risk remains much
smaller and follows the decreasing Hilbert--Schmidt scale up to a constant.
\textbf{(b)} The same quantities are plotted against the sample size \(n\)
for fixed \(d=5\) on logarithmic axes.  The parallel curves exhibit the
common \(n^{-1/2}\) rate and show that the difference between the estimators
lies in their spectral constants.  In both panels, circles and squares are
the Monte Carlo risks of \(\widehat{\operatorname{KSD}}_V\) and
\(\widehat{\operatorname{KSD}}_U\), respectively; the dashed curve is
\(\sqrt{\operatorname{tr}(C_\star)/n}\), and the dotted curve is
\(\sqrt{\|C_\star\|_{\mathrm{HS}}/n}\). Each marker averages 2,000
independent replications, and the error bars are 95\% Monte Carlo confidence
intervals.}
\label{fig:exp}
\end{figure}

\paragraph{Numerical illustration.}
We include a simulation at the Gaussian target to visualize the
trace-versus-Hilbert--Schmidt separation above.  We take
\(P_0=N(0,I_d)\), use the Gaussian kernel \(k_\gamma\) with
\(\gamma=\tfrac12\), and sample \(X_1,\dots,X_n\sim P_0\).  Since
\(\operatorname{KSD}(P_0,P_0)=0\), the risks displayed in
Figure~\ref{fig:exp} are
\[
  \mathbb E_{P_0^{\otimes n}}
  \widehat{\operatorname{KSD}}_V
  \quad\text{and}\quad
  \mathbb E_{P_0^{\otimes n}}
  \widehat{\operatorname{KSD}}_U.
\]
Each marker is an average over 2,000 independent Monte Carlo replications;
the error bars are 95\% Monte Carlo confidence intervals.  The trace and
Hilbert--Schmidt reference curves are evaluated using the exact Gaussian
formulas above.

The left panel shows the dimension dependence predicted by the explicit
Gaussian calculation.  The plug-in risk closely tracks the trace scale,
whereas the square-root U-statistic follows the smaller Hilbert--Schmidt
scale up to a spectral constant.  The right panel shows that both estimators
have \(n^{-1/2}\) sample-size behavior; their separation is therefore in
the spectral scale rather than in the power of \(n\).

\section{Conclusion}\label{sec:conclusion}

We proved that the minimax rate of estimating the kernel Stein
discrepancy at its target is of order 
\(
  \mathcal{O}\left(\sqrt{\|C_\star\|_{\mathrm{HS}}/n}\right)
\)
attained by the debiased square-root U-statistic but not by the plug-in
V-statistic, whose risk is $\asymp\sqrt{\operatorname{tr}(C_\star)/n}$ and
exceeds the optimum by $r_{\mathrm{eff}}(C_\star)^{1/4}$. For Gaussian
targets with a fixed-bandwidth Gaussian kernel this factor is exponential
in the dimension $d$.

The mechanism is the degeneracy of the Stein kernel at $P_0$ forced by
Stein's identity: the variance of $U_n$ collapses from $O(n^{-1})$ to
$O(n^{-2})$, and the square root turns this variance collapse into the Hilbert--Schmidt
risk scale. The plug-in cannot see this collapse, because its squared form retains the
diagonal terms $K_{P_0}(X_i,X_i)$ that keep it at the trace scale. The
scalar functional $\operatorname{KSD}(P_0,P)$ is thus easier to estimate
than the embedding $\mu_P$, a saving the plug-in does not exploit; for KSD
diagnostics in MCMC or variational inference we recommend the U-statistic
form, especially in high dimensions.

Our analysis assumes that the target score \(s_0=\nabla\log p_0\) is known. 
If the score is estimated, as in score matching or learned
energy-based models, the KSD estimator incurs an additional error not
covered by the present bounds, and identifying the sharp scale of this
error is a natural next question.

%% ---- Bibliography -------------------------------------------------
\bibliographystyle{unsrtnat}
\bibliography{references}

%% ====================================================================
\appendix
\section{Proof of Theorem~\ref{thm:ustat}}\label{app:proof-ustat}
Unbiasedness follows from the KSD identity:
\[
  \mathbb E_P K_{P_0}(X_1,X_2)
  =
  \|\mu_P\|_{\mathcal H}^2
  =
  \operatorname{KSD}^2(P_0,P).
\]

Let
\[
  h(x,y):=K_{P_0}(x,y).
\]
The variance formula for a second-order U-statistic 
\citep[][Eq.~5.13]{hoeffding1948class} gives
\[
  \operatorname{Var}_P(U_n)
  =
  \binom n2^{-1}
  \bigl[
  2(n-2)\zeta_1+\zeta_2
  \bigr],
\]
where
\[
  \zeta_1
  =
  \operatorname{Var}_P
  \left(
  \mathbb E[h(X,Y)\mid X]
  \right)
  =
  \operatorname{Var}_P
  \left(
  \left\langle
  \xi_{P_0}(X),\mu_P
  \right\rangle_{\mathcal H}
  \right),
\]
and
\[
  \zeta_2
  =
  \operatorname{Var}_P K_{P_0}(X,Y).
\]
Clearly,
\[
  \zeta_2
  \le
  V_P.
\]

It remains to bound \(\zeta_1\) in terms of \(V_P\).  Define the
second-moment operator
\[
  M_P
  :=
  \mathbb E_P
  \left[
  \xi_{P_0}(X)\otimes \xi_{P_0}(X)
  \right].
\]
Then
\[
  V_P
  =
  \mathbb E_{P\otimes P}
  \left\langle
  \xi_{P_0}(X),\xi_{P_0}(Y)
  \right\rangle_{\mathcal H}^2
  =
  \|M_P\|_{\mathrm{HS}}^2.
\]
Therefore,
\[
\begin{aligned}
  \zeta_1
  &\le
  \mathbb E_P
  \left\langle
  \xi_{P_0}(X),\mu_P
  \right\rangle_{\mathcal H}^2  \\
  &=
  \langle M_P\mu_P,\mu_P\rangle_{\mathcal H}  \\
  &\le
  \|M_P\|_{\mathrm{op}}\|\mu_P\|_{\mathcal H}^2  \\
  &\le
  \|M_P\|_{\mathrm{HS}}\theta^2  \\
  &=
  V_P^{1/2}\theta^2 .
\end{aligned}
\]
Thus
\[
  \operatorname{Var}_P(U_n)
  \le
  \frac{4\theta^2V_P^{1/2}}{n}
  +
  \frac{2V_P}{n(n-1)}.
\]

We next pass from the squared estimator to the KSD itself.  Since the map
\(u\mapsto\sqrt{u_+}\) is \(1/2\)-Hölder on \(\mathbb R\), for \(b\ge0\),
\[
  |\sqrt{a_+}-\sqrt b|
  \le
  \sqrt{|a-b|}.
\]
Applying this with \(a=U_n\) and \(b=\theta^2\), and then using Jensen's
inequality, gives
\[
  \mathbb E_P
  \left|
  \widehat{\operatorname{KSD}}_U-\theta
  \right|
  \le
  \bigl(\operatorname{Var}_P U_n\bigr)^{1/4}.
\]
If \(\theta>0\), we may instead use
\[
  |\sqrt{a_+}-\sqrt b|
  \le
  \frac{|a-b|}{\sqrt b},
  \qquad b>0.
\]
With \(b=\theta^2\), this gives
\[
  \mathbb E_P
  \left|
  \widehat{\operatorname{KSD}}_U-\theta
  \right|
  \le
  \frac{\mathbb E_P|U_n-\theta^2|}{\theta}
  \le
  \frac{\bigl(\operatorname{Var}_P U_n\bigr)^{1/2}}{\theta}.
\]

At \(P=P_0\), Stein's identity gives \(\theta=0\), hence \(\zeta_1=0\).
Moreover,
\[
  V_{P_0}
  =
  \mathbb E_{P_0\otimes P_0}K_{P_0}(X,Y)^2
  =
  \operatorname{tr}(C_\star^2).
\]
Therefore
\[
  \operatorname{Var}_{P_0}(U_n)
  \le
  \frac{2\operatorname{tr}(C_\star^2)}{n(n-1)},
\]
and \eqref{eq:ustat-null} follows.

\section{Proof of Proposition~\ref{prop:vlower}}\label{app:proof-vlower}

The upper bound was already shown in Theorem~\ref{thm:upper}. It remains
to prove the lower bound.
Write
\[
  \xi_i:=\xi_{P_0}(X_i),
  \qquad
  \widehat\mu_n=\frac1n\sum_{i=1}^n\xi_i,
  \qquad
  Z:=\|\widehat\mu_n\|.
\]
The variables $\xi_i$ are i.i.d.\ and centered by Stein's identity.  Hence,
using $\mathbb E\|\xi_i\|^2=\operatorname{tr}(C_\star)$ and independence,
\[
  \mathbb E Z^2
  =\frac1{n^2}\sum_{i,j=1}^n
    \mathbb E\langle\xi_i,\xi_j\rangle
  =\frac{\operatorname{tr}(C_\star)}{n}.
\]
We next compute the fourth moment.  Expanding
\[
  \left\|\sum_{i=1}^n\xi_i\right\|^4
  =\sum_{i,j,k,l=1}^n
    \langle\xi_i,\xi_j\rangle
    \langle\xi_k,\xi_l\rangle,
\]
independence and centering show that a term can be nonzero only if every
index appears at least twice.  The terms with all four indices equal
contribute $n\,\mathbb E_{P_0}K_{P_0}(X,X)^2$.  For two distinct indices,
the pairing $i=j$, $k=l$ contributes $(\operatorname{tr}C_\star)^2$, while
the other two pairings each contribute
$\mathbb E\langle\xi,\xi'\rangle^2=\operatorname{tr}(C_\star^2)$.
Consequently,
\[
  \mathbb E Z^4
  =\frac{n-1}{n^3}
    \left[(\operatorname{tr}C_\star)^2
      +2\operatorname{tr}(C_\star^2)\right]
    +\frac{\mathbb E_{P_0}K_{P_0}(X,X)^2}{n^3}.
\]
In particular, the assumed fourth moment makes $Z\in L^4$.  The
$L^p$ interpolation inequality
\(
  \|Z\|_{L^2}\le \|Z\|_{L^1}^{1/3}\|Z\|_{L^4}^{2/3}
\)
therefore yields
\[
  \mathbb E Z
  \ge
  \frac{(\mathbb E Z^2)^{3/2}}{(\mathbb E Z^4)^{1/2}}.
\]
Since
$\operatorname{tr}(C_\star^2)\le\operatorname{tr}(C_\star)^2$, the fourth
moment identity gives
\[
  \mathbb E Z^4
  \le
  \frac{3\operatorname{tr}(C_\star)^2}{n^2}
  +\frac{\mathbb E_{P_0}K_{P_0}(X,X)^2}{n^3}.
\]
For $n\ge n_0$ as defined in the proposition, this is at most
$4(\operatorname{tr}(C_\star)/n)^2$.  Substitution into the interpolation
bound proves
\[
  \mathbb E_{P_0}\widehat{\operatorname{KSD}}_V
  =\mathbb E Z
  \ge
  \frac12\sqrt{\frac{\operatorname{tr}(C_\star)}{n}},
\]
as required.

\section{Gaussian target and Gaussian kernel}\label{app:gauss}

\begin{proposition}[Stein kernel for Gaussian target]\label{prop:gausskernel}
For $P_0=N(0,I_d)$, $s_0(x)=-x$, and
\[
  K_{P_0}(x,y)
  =e^{-\gamma\|x-y\|_2^2}
   \bigl[x^\top y+2\gamma d-(2\gamma+4\gamma^2)\|x-y\|_2^2\bigr],
  \quad
  K_{P_0}(x,x)=\|x\|_2^2+2\gamma d.
\]
\end{proposition}
\begin{proof}
Write $r=x-y$, $k=e^{-\gamma\|r\|^2}$. Then $\nabla_x k=-2\gamma r\,k$,
$\nabla_y k=2\gamma r\,k$, $\nabla_x\!\cdot\!\nabla_y k
=2\gamma d\,k-4\gamma^2\|r\|^2 k$. Substituting $s_0(x)=-x$,
$s_0(y)=-y$ into \eqref{eq:steinkernel} and collecting terms gives the
claim; setting $y=x$ gives $K_{P_0}(x,x)$.
\end{proof}

\begin{proposition}[Spectral functionals]\label{prop:traces}
Under $P_\star=N(0,I_d)$, $\operatorname{tr}(C_\star)=(1+2\gamma)d$, and
\[
  \operatorname{tr}(C_\star^2)
  =(1+8\gamma)^{-d/2}\frac{d}{(1+8\gamma)^2}
  \Bigl[(64\gamma^4+32\gamma^3+4\gamma^2)d
    +128\gamma^4+128\gamma^3+80\gamma^2+16\gamma+1\Bigr].
\]
\end{proposition}
\begin{proof}
With $X,Y\sim N(0,I_d)$ independent put $U=X-Y$, $V=X+Y$
(independent $N(0,2I_d)$), $a=\|U\|^2\sim2\chi^2_d$,
$b=\|V\|^2\sim2\chi^2_d$. Then $\|X-Y\|^2=a$,
$X^\top Y=\tfrac{1}{4}(b-a)$, and by Proposition~\ref{prop:gausskernel},
$K_{P_0}(X,Y)=e^{-\gamma a}Q$ with $Q=\tfrac{b}{4}-\beta a+2\gamma d$,
$\beta=(2\gamma+\tfrac{1}{2})^2$. Conditioning on $a$:
$\mathbb{E}[b\mid a]=2d$, $\mathbb{E}[b^2\mid a]=4d^2+8d$, so
$\mathbb{E}[Q^2\mid a]=\alpha_0+\alpha_1 a+\alpha_2 a^2$ with
\[
  \alpha_0=d(\tfrac{1}{2}+d\beta),
  \quad
  \alpha_1=-\beta d(1+4\gamma),
  \quad
  \alpha_2=\beta^2.
\]
Evaluating via the $\chi^2$ MGF $\mathbb{E}e^{ta}=(1-4t)^{-d/2}$ at
$t=-2\gamma$:
\[
  M_0=(1+8\gamma)^{-d/2},\quad
  M_1=\tfrac{2d}{1+8\gamma}M_0,\quad
  M_2=\tfrac{4d(d+2)}{(1+8\gamma)^2}M_0.
\]
Then $\operatorname{tr}(C_\star^2)
=\alpha_0M_0+\alpha_1M_1+\alpha_2M_2$.
Factoring $(1+8\gamma)^{-d/2}d/(1+8\gamma)^2$ and applying the identities
$16\gamma^2\beta=64\gamma^4+32\gamma^3+4\gamma^2$ and
$\tfrac{1}{2}(1+8\gamma)^2+8\beta^2
=128\gamma^4+128\gamma^3+80\gamma^2+16\gamma+1$
yields Proposition~\ref{prop:traces}.

\end{proof}

\begin{corollary}[Dimension dependence of the gap]\label{cor:dim}
The plug-in scale is $\sqrt{(1+2\gamma)d/n}$.
\begin{enumerate}
\item \emph{(Fixed bandwidth.)} For fixed $\gamma>0$,
  $\operatorname{tr}(C_\star^2)\asymp_\gamma d^2(1+8\gamma)^{-d/2}$,
  the optimal scale is $\asymp_\gamma d^{1/2}(1+8\gamma)^{-d/8}/\sqrt{n}$,
  and the gap $r_{\mathrm{eff}}^{1/4}\asymp_\gamma(1+8\gamma)^{d/8}$
  is \emph{exponential in $d$}.
\item \emph{(Scaled bandwidth $\gamma=\alpha/d$).}
  $\operatorname{tr}(C_\star^2)\asymp_\alpha d$, the optimal scale is
  $\asymp_\alpha d^{1/4}/\sqrt{n}$, and the gap
  $r_{\mathrm{eff}}^{1/4}\asymp_\alpha d^{1/4}$ is polynomial in $d$.
\end{enumerate}
\end{corollary}
\begin{proof}
The bracket in Proposition~\ref{prop:traces} is $c_1(\gamma)d+c_0(\gamma)$
with $c_1(\gamma)=4\gamma^2(4\gamma+1)^2$.
\emph{(i)} The $c_1 d$ term dominates, giving
$\operatorname{tr}(C_\star^2)\asymp d^2(1+8\gamma)^{-d/2}$ and
$r_{\mathrm{eff}}^{1/4}=(d^2/(d^2(1+8\gamma)^{-d/2}))^{1/4}
=(1+8\gamma)^{d/8}$.
\emph{(ii)} For $\gamma=\alpha/d$: $(1+8\alpha/d)^{-d/2}\to e^{-4\alpha}$,
$c_1(\alpha/d)d\to0$, so $\operatorname{tr}(C_\star^2)\to e^{-4\alpha}d$
and $r_{\mathrm{eff}}^{1/4}\asymp(d^2/d)^{1/4}=d^{1/4}$.
\end{proof}

\section{Proof of Theorem~\ref{thm:lower}}\label{app:proof-lower}
We use the method of two fuzzy hypotheses, in the form of
\citet[Section~2.7.4]{tsybakov2009introduction}.  In the present setting,
the parameter is the probability measure \(P\), the observation law
associated with \(P\) is \(P^{\otimes n}\), and the functional to be
estimated is
\[
  F(P):=\operatorname{KSD}(P_0,P).
\]

Write \(P_\star=P_0\).  Let \((\lambda_j,e_j)_{j\ge1}\) be the nonzero
eigenpairs of \(C_\star\).  For \(\lambda_j>0\), define
\[
  g_j(x)
  :=
  \frac{
  \langle \xi_{P_0}(x),e_j\rangle_{\mathcal H}
  }{\sqrt{\lambda_j}} .
\]
Then
\[
  \mathbb E_{P_\star} g_j=0,
  \qquad
  \mathbb E_{P_\star} g_i g_j=\delta_{ij},
  \qquad
  \mathbb E_{P_\star}
  \bigl[g_j(X)\xi_{P_0}(X)\bigr]
  =
  \sqrt{\lambda_j}\,e_j .
\]
Indeed, the first identity follows from Stein's identity
\(\mu_{P_\star}=0\), the second from
\[
  \langle C_\star e_i,e_j\rangle_{\mathcal H}
  =
  \lambda_i\delta_{ij},
\]
and the third from the definition of \(C_\star\).

Let
\[
  S:=\operatorname{tr}(C_\star^2)>0,
  \qquad
  S_m:=\sum_{j=1}^m\lambda_j^2.
\]
Since \(S_m\uparrow S\), let \(m\) be the smallest integer such that
\[
  S_m\ge \frac12 S.
\]
We fix this \(m\) throughout the construction.

We first give the argument under the additional assumption that, for the
fixed integer \(m\) used below, the coordinates \(g_1,\dots,g_m\) are
bounded.  For
\(\varepsilon=(\varepsilon_1,\dots,\varepsilon_m)\in\{-1,1\}^m\), define
\[
  f_\varepsilon
  =
  1+
  \sum_{j=1}^m
  \frac{a_j\varepsilon_j}{\sqrt{\lambda_j}}g_j,
  \qquad
  dP_\varepsilon=f_\varepsilon\,dP_\star ,
\]
where the amplitudes \(a_j\ge0\) will be chosen below.  If
\begin{equation}
  \label{eq:bounded-likelihood}
  \sum_{j=1}^m
  \frac{a_j}{\sqrt{\lambda_j}}
  \|g_j\|_{L^\infty(P_\star)}
  \le 1,
\end{equation}
then \(0\le f_\varepsilon\le2\) for every \(\varepsilon\).  Since
\(\mathbb E_{P_\star}g_j=0\), we also have
\[
  \int f_\varepsilon\,dP_\star=1.
\]
Thus each \(P_\varepsilon\) is a probability measure and satisfies the
bounded likelihood-ratio condition
\[
0 \le \frac{dP_\varepsilon}{dP_\star}\le 2,
\]
therefore, by Remark~\ref{rem: bounded likelihood}, \(P_\varepsilon\in\mathcal P(A)\) for every \(A\ge4\).

We now define the two fuzzy hypotheses.  Let
\[
  \mu_0=\delta_{P_\star},
  \qquad
  \mu_1
  =
  2^{-m}
  \sum_{\varepsilon\in\{-1,1\}^m}
  \delta_{P_\varepsilon}.
\]
The corresponding mixture experiments are
\[
  \mathbb P_0
  =
  P_\star^{\otimes n},
  \qquad
  \mathbb P_1
  =
  2^{-m}
  \sum_{\varepsilon\in\{-1,1\}^m}
  P_\varepsilon^{\otimes n}.
\]

We next verify the separation condition.  For each
\(\varepsilon\in\{-1,1\}^m\),
\[
  \mu_{P_\varepsilon}
  =
  \mathbb E_{P_\star}
  \left[
  f_\varepsilon(X)\xi_{P_0}(X)
  \right]
  =
  \sum_{j=1}^m a_j\varepsilon_j e_j .
\]
Therefore
\[
  F(P_\varepsilon)
  =
  \operatorname{KSD}(P_0,P_\varepsilon)
  =
  \|\mu_{P_\varepsilon}\|_{\mathcal H}
  =
  \left(\sum_{j=1}^m a_j^2\right)^{1/2}
  =: r,
\]
whereas
\[
  F(P_\star)=\operatorname{KSD}(P_0,P_\star)=0.
\]
Hence the separation condition in the fuzzy-hypothesis theorem holds (see Theorem 2.14, assumption (i) in \cite{tsybakov2009introduction}) with
\[
  c=0,
  \qquad
  s=\frac r2,
  \qquad
  \beta_0=\beta_1=0.
\]

It remains to control the distance between \(\mathbb P_1\) and
\(\mathbb P_0\).  The likelihood ratio of \(\mathbb P_1\) with respect to
\(\mathbb P_0\) is
\[
  L
  =
  \frac{d\mathbb P_1}{d\mathbb P_0}
  =
  2^{-m}
  \sum_{\varepsilon\in\{-1,1\}^m}
  \prod_{i=1}^n f_\varepsilon(X_i).
\]
Thus
\[
  1+\chi^2(\mathbb P_1,\mathbb P_0)
  =
  \mathbb E_{\varepsilon,\varepsilon'}
  \left[
  \mathbb E_{P_\star}
  f_\varepsilon(X)f_{\varepsilon'}(X)
  \right]^n .
\]
By orthonormality of \(g_1,\dots,g_m\),
\[
  \mathbb E_{P_\star}
  f_\varepsilon f_{\varepsilon'}
  =
  1+
  \sum_{j=1}^m
  \frac{a_j^2}{\lambda_j}
  \varepsilon_j\varepsilon'_j .
\]
Let
\[
  \eta_j=\varepsilon_j\varepsilon'_j .
\]
Then \(\eta_1,\dots,\eta_m\) are independent Rademacher random variables, and
\[
  1+\chi^2(\mathbb P_1,\mathbb P_0)
  =
  \mathbb E_{\eta}
  \left[
  \left(
  1+
  \sum_{j=1}^m
  \frac{a_j^2}{\lambda_j}\eta_j
  \right)^n
  \right].
\]
The quantity inside the power is nonnegative, because it is an expectation
of a product of two nonnegative densities.  Therefore, using
\((1+w)^n\le \exp(nw)\) whenever \(1+w\ge0\), we obtain
\[
  1+\chi^2(\mathbb P_1,\mathbb P_0)
  \le
  \mathbb E_{\eta}
  \exp\left\{
  n
  \sum_{j=1}^m
  \frac{a_j^2}{\lambda_j}\eta_j
  \right\}.
\]
By independence of the \(\eta_j\)'s,
\[
  1+\chi^2(\mathbb P_1,\mathbb P_0)
  \le
  \prod_{j=1}^m
  \cosh\left(
  \frac{n a_j^2}{\lambda_j}
  \right).
\]
Since \(\cosh x\le \exp(x^2/2)\), it follows that
\[
  1+\chi^2(\mathbb P_1,\mathbb P_0)
  \le
  \exp\left\{
  \frac12
  \sum_{j=1}^m
  \left(
  \frac{n a_j^2}{\lambda_j}
  \right)^2
  \right\}.
\]

Choose
\[
  c_0=2\log(5/4)
\]
and impose
\[
  \sum_{j=1}^m
  \left(
  \frac{n a_j^2}{\lambda_j}
  \right)^2
  \le c_0.
\]
Then
\[
  \chi^2(\mathbb P_1,\mathbb P_0)
  \le
  e^{c_0/2}-1
  =
  \frac14 .
\]
Therefore, by the \(\chi^2\)-version of the fuzzy-hypothesis theorem (see Theorem 2.15, part (iii) of \cite{tsybakov2009introduction}), there
exists a numerical constant \(\eta>0\) such that every estimator
\(\widehat T\) satisfies
\[
  \sup_{P\in\{P_\star\}\cup\{P_\varepsilon:
  \varepsilon\in\{-1,1\}^m\}}
  P^{\otimes n}
  \left(
  \left|
  \widehat T-F(P)
  \right|
  \ge \frac r2
  \right)
  \ge
  \eta .
\]
Consequently,
\[
  \sup_{P\in\{P_\star\}\cup\{P_\varepsilon:
  \varepsilon\in\{-1,1\}^m\}}
  \mathbb E_P
  \left|
  \widehat T-F(P)
  \right|
  \ge
  \frac r2\,\eta .
\]

We now choose the amplitudes.  
Set
\[
  a_j^2
  =
  \frac{\sqrt{c_0}}{n}
  \frac{\lambda_j^2}{S_m^{1/2}},
  \qquad
  j=1,\dots,m .
\]
Then
\[
  \sum_{j=1}^m
  \left(
  \frac{n a_j^2}{\lambda_j}
  \right)^2
  =
  c_0,
\]
and
\[
  r
  =
  \left(\sum_{j=1}^m a_j^2\right)^{1/2}
  =
  c_0^{1/4}
  \frac{S_m^{1/4}}{\sqrt n}.
\]
The boundedness condition \ref{eq:bounded-likelihood} is satisfied for all sufficiently~
large \(n\), because
\[
  \sum_{j=1}^m
  \frac{a_j}{\sqrt{\lambda_j}}
  \|g_j\|_{L^\infty(P_\star)}
  =
  \mathcal{O}(n^{-1/2}).
\]
Hence
\[
  \inf_{\widehat T}
  \sup_{P\in\mathcal P(A)}
  \mathbb E_P
  \left|
  \widehat T-\operatorname{KSD}(P_0,P)
  \right|
  \ge
  c_1
  \frac{S_m^{1/4}}{\sqrt n} \ge c \frac{S^{1/4}}{\sqrt n},
\]
where \(c >0\) is a numerical constant.

This proves the theorem in the bounded-coordinate case.

\paragraph{Removing boundedness assumption on $g_j$.}
It remains to remove the boundedness assumption.  Fix \(m\).  For
\(\tau>0\), define
\[
  \widetilde g_j^\tau
  =
  g_j\mathbf 1\{|g_j|\le \tau\},
  \qquad
  \bar g_j^\tau
  =
  \widetilde g_j^\tau
  -
  \mathbb E_{P_\star}\widetilde g_j^\tau .
\]
Since \(g_j\in L^2(P_\star)\), we have
\[
  \bar g_j^\tau\to g_j
  \qquad\text{in }L^2(P_\star)
\]
for each fixed \(j\).  Let \(G_\tau\) be the \(m\times m\) Gram matrix of
\[
  \bar g_1^\tau,\dots,\bar g_m^\tau
\]
in \(L^2(P_\star)\).  Since \(g_1,\dots,g_m\) are orthonormal,
\[
  G_\tau\to I_m.
\]
Thus, for all large enough \(\tau\), \(G_\tau\) is invertible.  Define
\[
  \phi^\tau
  =
  G_\tau^{-1/2}\bar g^\tau,
  \qquad
  \bar g^\tau
  =
  (\bar g_1^\tau,\dots,\bar g_m^\tau)^\top .
\]
Then \(\phi_1^\tau,\dots,\phi_m^\tau\) are bounded, zero-mean, and
orthonormal in \(L^2(P_\star)\).  Moreover, for fixed \(m\),
\[
  \max_{1\le j\le m}
  \|\phi_j^\tau-g_j\|_{L^2(P_\star)}
  \to0,
  \qquad
  \max_{1\le j\le m}
  \|\phi_j^\tau\|_{L^\infty(P_\star)}
  \le C_m\tau
\]
for all sufficiently large \(\tau\).

Repeat the preceding construction with \(\phi_j^\tau\) in place of \(g_j\):
\[
  f_\varepsilon^\tau
  =
  1+
  \sum_{j=1}^m
  \frac{a_j\varepsilon_j}{\sqrt{\lambda_j}}
  \phi_j^\tau,
  \qquad
  dP_\varepsilon^\tau
  =
  f_\varepsilon^\tau\,dP_\star .
\]
Choose \(\tau=\tau_n\) such that
\[
  \tau_n\to\infty,
  \qquad
  \frac{\tau_n}{\sqrt n}\to0.
\]
For instance, take \(\tau_n=n^{1/4}\).  Since \(a_j=\mathcal{O}_m(n^{-1/2})\),
\[
  \sum_{j=1}^m
  \frac{a_j}{\sqrt{\lambda_j}}
  \|\phi_j^{\tau_n}\|_{L^\infty(P_\star)}
  \le
  C_m\frac{\tau_n}{\sqrt n}
  \to0.
\]
Thus \(0\le f_\varepsilon^{\tau_n}\le2\) for all sufficiently large \(n\), therefore \(P_\varepsilon^{\tau_n}\in\mathcal P(A)\) for all sufficiently large \(n\).

The chi-square calculation is unchanged, because
\(\phi_1^{\tau_n},\dots,\phi_m^{\tau_n}\) are zero-mean and orthonormal in
\(L^2(P_\star)\).  The only difference is the KSD separation.  Define
\[
  u_j^\tau
  :=
  \mathbb E_{P_\star}
  \left[
  \phi_j^\tau(X)\xi_{P_0}(X)
  \right].
\]
Then
\[
  u_j^\tau\to \sqrt{\lambda_j}e_j
  \qquad\text{in }\mathcal H .
\]
Indeed, for every \(h\in L^2(P_\star)\),
\[
  \left\|
  \mathbb E_{P_\star}
  \left[
  h(X)\xi_{P_0}(X)
  \right]
  \right\|_{\mathcal H}
  \le
  \|h\|_{L^2(P_\star)}
  \|C_\star\|_{\mathrm{op}}^{1/2}.
\]
Applying this inequality with \(h=\phi_j^\tau-g_j\) gives the claim.

Therefore, uniformly over \(\varepsilon\in\{-1,1\}^m\),
\[
  \mu_{P_\varepsilon^{\tau_n}}
  =
  \sum_{j=1}^m
  \frac{a_j\varepsilon_j}{\sqrt{\lambda_j}}
  u_j^{\tau_n}
  =
  \sum_{j=1}^m
  a_j\varepsilon_j e_j
  +
  o(n^{-1/2})
\]
in \(\mathcal H\).  Since
\[
  r
  =
  c_0^{1/4}
  \frac{S_m^{1/4}}{\sqrt n},
\]
we obtain
\[
  \operatorname{KSD}(P_0,P_\varepsilon^{\tau_n})
  =
  r(1+o(1))
\]
uniformly in \(\varepsilon\).  Hence, for all sufficiently large \(n\),
\[
  \operatorname{KSD}(P_0,P_\varepsilon^{\tau_n})
  \ge
  \frac r2 .
\]
Thus the separation condition in the fuzzy-hypothesis theorem holds with
\(c=0\), \(\beta_0=\beta_1=0\), and \(s=r/4\).  The same chi-square bound
therefore gives
\[
  \inf_{\widehat T}
  \sup_{P\in\mathcal P(A)}
  \mathbb E_P
  \left|
  \widehat T-\operatorname{KSD}(P_0,P)
  \right|
  \ge
  c
  \frac{S_m^{1/4}}{\sqrt n}
\]
for all sufficiently large \(n\).  Since \(S_m\ge S/2=\operatorname{tr}(C_\star^2)/2\), this completes the proof.

\end{document}

%% file: math_commands.tex
\usepackage{amsmath,amsfonts,bm,amsthm}

\def\eqref#1{equation~\ref{#1}}
\def\1{\bm{1}}

\DeclareMathAlphabet{\mathsfit}{\encodingdefault}{\sfdefault}{m}{sl}
\SetMathAlphabet{\mathsfit}{bold}{\encodingdefault}{\sfdefault}{bx}{n}

\newcommand{\R}{\mathbb{R}}

\newtheorem{theorem}{Theorem}
 
\newtheorem{proposition}[theorem]{Proposition} 
\newtheorem{remark}[theorem]{Remark}
\newtheorem{corollary}[theorem]{Corollary}

\newtheorem{assumption}{Assumption}

%% file: references.bib
@article{kalinke2024nystr,
  title={Nystr{\"o}m Kernel Stein Discrepancy},
  author={Kalinke, Florian and Szab{\'o}, Zolt{\'a}n and Sriperumbudur, Bharath K},
  journal={arXiv preprint arXiv:2406.08401},
  year={2024}
}

@inproceedings{
balasubramanian2025improved,
title={Improved Finite-Particle Convergence Rates for Stein Variational Gradient Descent},
author={Krishna Balasubramanian and Sayan Banerjee and Promit Ghosal},
booktitle={The Thirteenth International Conference on Learning Representations},
year={2025}
}

@article{kalinke2026nystr,
  title={Nystr{\"o}m Kernel Stein Discrepancy Tests},
  author={Kalinke, Florian and Szab{\'o}, Zolt{\'a}n and Sriperumbudur, Bharath K},
  journal={arXiv preprint arXiv:2605.25173},
  year={2026}
}

@article{liu2026probabilistic,
  title={Probabilistic Inference and Learning with Stein's Method},
  author={Liu, Qiang and Mackey, Lester and Oates, Chris},
  journal={arXiv preprint arXiv:2603.07467},
  year={2026}
}

@article{hagrass2026minimax,
  title={Minimax optimal goodness-of-fit testing with kernel Stein discrepancy},
  author={Hagrass, Omar and Sriperumbudur, Bharath and Balasubramanian, Krishnakumar},
  journal={Bernoulli},
  volume={32},
  number={1},
  pages={299--324},
  year={2026},
  publisher={Bernoulli Society for Mathematical Statistics and Probability}
}

@article{birge1995estimation,
  title={Estimation of Integral Functionals of a Density},
  author={Birg{\'e}, Lucien and Massart, Pascal},
  journal={The Annals of Statistics},
  volume={23},
  number={1},
  year={1995},
  pages={11--29}
}

@article{gine2008simple,
  title={A simple adaptive estimator of the integrated square of a density},
  author={Gin{\'e}, Evarist and Nickl, Richard},
  journal={Bernoulli},
  volume={14},
  number={1},
  year={2008},
  pages={47--61}
}

@inproceedings{barp2019minimum,
  title={Minimum {S}tein discrepancy estimators},
  author={Barp, Alessandro and Briol, Fran{\c c}ois-Xavier and Duncan, Andrew B. and Girolami, Mark and Mackey, Lester},
  booktitle={Advances in Neural Information Processing Systems (NeurIPS)},
  volume={32},
  year={2019}
}

@inproceedings{chen2018stein,
  title={Stein points},
  author={Chen, Wilson Ye and Mackey, Lester and Gorham, Jackson and Briol, Fran{\c c}ois-Xavier and Oates, Chris J.},
  booktitle={International Conference on Machine Learning (ICML)},
  pages={844--853},
  year={2018}
}

@article{kanagawa2019kernel,
  title={A kernel {S}tein test for comparing latent variable models},
  author={Kanagawa, Heishiro and Jitkrittum, Wittawat and Mackey, Lester and Fukumizu, Kenji and Gretton, Arthur},
  journal={Journal of the Royal Statistical Society Series B: Statistical Methodology},
  volume={85},
  number={3},
  pages={986--1011},
  year={2023}
}

@inproceedings{huggins2018random,
  title={Random feature {S}tein discrepancies},
  author={Huggins, Jonathan H. and Mackey, Lester},
  booktitle={Advances in Neural Information Processing Systems (NeurIPS)},
  volume={31},
  year={2018}
}

@inproceedings{jitkrittum2017linear,
  title={A linear-time kernel goodness-of-fit test},
  author={Jitkrittum, Wittawat and Xu, Wenkai and Szab{\'o}, Zolt{\'a}n and Fukumizu, Kenji and Gretton, Arthur},
  booktitle={Advances in Neural Information Processing Systems (NeurIPS)},
  volume={30},
  year={2017}
}

@article{tolstikhin2017minimax,
  title={Minimax estimation of kernel mean embeddings},
  author={Tolstikhin, Ilya and Sriperumbudur, Bharath K and Muandet, Krikamol},
  journal={Journal of Machine Learning Research},
  volume={18},
  number={86},
  pages={1--47},
  year={2017}
}

@inproceedings{stein1972bound,
  title={A bound for the error in the normal approximation to the distribution of a sum of dependent random variables},
  author={Stein, Charles},
  booktitle={Proceedings of the Sixth Berkeley Symposium on Mathematical Statistics and Probability, Volume 2: Probability Theory},
  pages={583--602},
  year={1972},
  publisher={University of California Press}
}

@book{stein1986approximate,
  title={Approximate Computation of Expectations},
  author={Stein, Charles},
  year={1986},
  publisher={Institute of Mathematical Statistics}
}

@book{chen2011normal,
  title={Normal Approximation by Stein's Method},
  author={Chen, Louis H. Y. and Goldstein, Larry and Shao, Qi-Man},
  year={2011},
  publisher={Springer}
}

@article{hoeffding1948class,
  title   = {A Class of Statistics with Asymptotically Normal Distribution},
  author  = {Hoeffding, Wassily},
  journal = {The Annals of Mathematical Statistics},
  volume  = {19},
  number  = {3},
  pages   = {293--325},
  year    = {1948}
}

@inproceedings{
cribeiro2025minimax,
title={The Minimax Lower Bound of Kernel Stein Discrepancy Estimation},
author={Cribeiro-Ramallo, Jose and Aich, Agnideep and Kalinke, Florian and Aich, Ashit Baran and Szab{\'o}, Zolt{\'a}n},
booktitle={The 29th International Conference on Artificial Intelligence and Statistics},
year={2026}
}

@inproceedings{liu2016kernelized,
  title={A kernelized {S}tein discrepancy for goodness-of-fit tests},
  author={Liu, Qiang and Lee, Jason D. and Jordan, Michael I.},
  booktitle={International Conference on Machine Learning (ICML)},
  pages={276--284},
  year={2016}
}

@inproceedings{chwialkowski2016kernel,
  title={A kernel test of goodness of fit},
  author={Chwialkowski, Kacper and Strathmann, Heiko and Gretton, Arthur},
  booktitle={International Conference on Machine Learning (ICML)},
  pages={2606--2615},
  year={2016}
}

@inproceedings{gorham2017measuring,
  title={Measuring sample quality with kernels},
  author={Gorham, Jackson and Mackey, Lester},
  booktitle={International Conference on Machine Learning (ICML)},
  pages={1292--1301},
  year={2017}
}

@inproceedings{gorham2015measuring,
  title={Measuring sample quality with {S}tein's method},
  author={Gorham, Jackson and Mackey, Lester},
  booktitle={Advances in Neural Information Processing Systems (NeurIPS)},
  volume={28},
  year={2015}
}

@article{muandet2017kernel,
  title={Kernel mean embedding of distributions: A review and beyond},
  author={Muandet, Krikamol and Fukumizu, Kenji and Sriperumbudur, Bharath and Sch{\"o}lkopf, Bernhard},
  journal={Foundations and Trends in Machine Learning},
  volume={10},
  number={1--2},
  pages={1--141},
  year={2017}
}

@book{tsybakov2009introduction,
  title={Introduction to Nonparametric Estimation},
  author={Tsybakov, Alexandre B.},
  year={2009},
  publisher={Springer}
}

@inproceedings{tolstikhin2016minimax,
  title={Minimax estimation of maximum mean discrepancy with radial kernels},
  author={Tolstikhin, Ilya O. and Sriperumbudur, Bharath K. and Sch{\"o}lkopf, Bernhard},
  booktitle={Advances in Neural Information Processing Systems (NeurIPS)},
  volume={29},
  pages={1930--1938},
  year={2016}
}

@article{gretton2012kernel,
  title={A kernel two-sample test},
  author={Gretton, Arthur and Borgwardt, Karsten M. and Rasch, Malte J. and Sch{\"o}lkopf, Bernhard and Smola, Alexander},
  journal={Journal of Machine Learning Research},
  volume={13},
  pages={723--773},
  year={2012}
}

@inproceedings{liu2016stein,
  title={Stein variational gradient descent: A general purpose {B}ayesian inference algorithm},
  author={Liu, Qiang and Wang, Dilin},
  booktitle={Advances in Neural Information Processing Systems (NeurIPS)},
  year={2016}
}
